\documentclass[12pt]{amsart}
\newtheorem{thm}{Theorem}[section]
\newtheorem{lem}[thm]{Lemma}

\newtheorem{conj}[thm]{Conjecture}
\theoremstyle{remark}

\newtheorem{rem}[thm]{Remark}

\newtheorem*{prfofthm1}{Proof of Theorem~\ref{Mainthm}}

\newtheorem*{acknowledgement}{Acknowledgment}
\makeatletter
 
 \@addtoreset{equation}{section}

\makeatother

\title{Durfee-type inequality for complete intersection surface singularities}
\author{Makoto Enokizono}
\subjclass[2010]{14D06}
\thanks{
	{\bf Keywords:}
fibered surface, singularity, complete intersection}
\address{Makoto Enokizono,
	Department of Mathematics,
	Graduate School of Science,
	Osaka University,
	Toyonaka, Osaka 560-0043, Japan}
\email{m-enokizono@cr.math.sci.osaka-u.ac.jp}
\usepackage[top=35truemm,bottom=35truemm,left=25truemm,right=25truemm]{geometry}
\usepackage{amsmath,cases}
\usepackage{amsfonts}
\usepackage[all]{xy}
\usepackage{here}
\usepackage{latexsym}
  
\begin{document}
\maketitle

\begin{abstract}
We prove that the signature of the Milnor fiber of smoothings of a $2$-dimensional isolated complete intersection singularity does not exceed the negative number determined by the geometric genus, the embedding dimension and the number of irreducible components of the exceptional set of the minimal resolution, which implies  Durfee's weak conjecture and a partial answer to Kerner--N\'emethi's conjecture.
\end{abstract}

\section*{Introduction}

Let $(X,o)$ be an analytic germ of a $2$-dimensional complex isolated complete intersection singularity with embedding dimension $n\ge 3$.
As analytic germs, it is isomorphic to the central fiber $o\in h^{-1}(0)$ of 
some holomorphic map-germ $h=(h_1,\ldots h_{n-2})\colon (\mathbb{C}^n,o) \to (\mathbb{C}^{n-2},0)$. 
We put $X_{\delta}=h^{-1}(\delta)$ for $\delta\in \mathbb{C}^{n-2}$ and $B_{\varepsilon}\subset \mathbb{C}^{n}$ the closed ball of radius $\varepsilon$ centered at $o$.
For a regular value $\delta\in \mathbb{C}^{n-2}$ and $\varepsilon>0$, $M=X_{\delta}\cap B_{\varepsilon}$ is a real $4$-dimensional manifold with boundary and its topology is independent of the choices of $\delta$ and $\varepsilon$ with $||\delta||>0$ and $\varepsilon>0$ sufficiently small, which is called the {\em Milnor fiber}.
The second Betti number $\mu:=\mathrm{rank}H_2(M,\mathbb{Z})$ is called the {\em Milnor number}.
Let $\mu_{+}$ (resp.\ $\mu_{-}$, $\mu_0$) be the number of positive (resp.\ negative, 0) eigenvalues of the natural intersection form $H_2(M,\mathbb{Z})\times H_2(M,\mathbb{Z})\to \mathbb{Z}$.
Then we have $\mu=\mu_{+}+\mu_{-}+\mu_0$ and call $\sigma:=\mu_{+}-\mu_{-}$ the {\em signature} of the Milnor fiber $M$.
In \cite{Dur}, Durfee conjectured that the signature $\sigma$ is always non-positive, which is nowadays called Durfee's weak conjecture (there is a counterexample to the conjecture for the non-complete intersection case, cf.\ \cite{Wah}).
Our main theorem in this paper is that Durfee's weak conjecture is true for all complete intersection surface singularities.
More precisely, we prove the following:

\begin{thm}[Theorem~\ref{Signthm}] \label{signthm}
Let $(X,o)$ be a germ of an isolated complete intersection surface singularity with embedding dimension $n$.
Let $\pi\colon \widetilde{X}\to X$ be the minimal resolution of $(X,o)$ and $p_g:=\mathrm{dim}_{\mathbb{C}}R^{1}\pi_{*}\mathcal{O}_{\widetilde{X}}$ the geometric genus of $(X,o)$.
Then we have
$$
\sigma\le -\frac{8}{3n-5}p_g-\#\mathrm{exc}(\pi)
$$
or equivalently,
\begin{equation} \label{pgmueq}
\frac{12(n-1)}{3n-5}p_g\le \mu+1-\chi_{\mathrm{top}}(\mathrm{exc}(\pi))
\end{equation}
with the equality holding if and only if $(X,o)$ is a rational double point,
where $\#\mathrm{exc}(\pi)$ is the number of irreducible components of the exceptional set $\mathrm{exc}(\pi)$ and $\chi_{\mathrm{top}}(\mathrm{exc}(\pi))$ its topological Euler number.
\end{thm}

Theorem~\ref{signthm} is a generalization of the Durfee-type inequality for hypersurface surface singularities established in \cite{Eno}.
The coefficient of $p_g$ in \eqref{pgmueq} coincides with the number $C_{2,n-2}$ defined in \cite{KeNe}.
Kerner and N\'emethi conjectured in \cite{KeNe} that for any isolated complete intersection singularity of dimension $n$ and embedding dimension $n+r$ with $n\ge 3$ or $(n,r)=(2,1)$, one has $C_{n,r}p_g\le \mu$, where $C_{n,r}$ is the rational number determined by $n$ and $r$ (more details, see \cite{KeNe}).
They also give counterexamples of complete intersection singularities for the case of $n=2$ and $r>1$.
Nevertheless, Theorem~\ref{signthm} says that Kerner-N\'emethi's conjecture ``asymptotically'' holds for $n=2$ and arbitrary $r$.
Note that $C_{n,1}=6$ and Durfee's strong conjecture is nothing but Kerner-Nemethi's conjecture for the case of $n=2$ and $r=1$, which remains open.
Durfee's weak conjecture for the hypersurface case was first proved by Koll\'ar and N\'emethi in \cite{KoNe}

To prove Theorem~\ref{signthm}, we use the method of invariants of fibered surfaces.
A {\em fibered surface} (or a {\em fibration} for short) $f\colon S\to B$ is a surjective morphism from a non-singular projective surface $S$ to a non-singular projective curve $B$ with connected fibers.
Let $K_f=K_S-f^{*}K_B$ denote the relative canonical bundle of $f$ and put $\chi_f:=\mathrm{deg}f_*\mathcal{O}(K_f)$.
The ratio $K_f^2/\chi_f$ of the self-intersection number $K_f^2$ and $\chi_f$ is called the {\em slope} of $f$.
In \cite{Eno}, the lower bound of the slope of fibered surfaces whose general fiber is a plane curve was established. 
As an application, a Durfee-type inequality holds for $2$-dimensional hypersurface singularities.
In this paper, we consider fibered surfaces whose general fiber is a complete intersection curve in the projective space and establish the lower bound of the slope of such fibrations under additional assumptions:

\begin{thm}[Theorem~\ref{Mainthm}] \label{mainthm}
Let $f\colon S\to B$ be a relatively minimal fibration whose general fiber $F$ is a complete intersection of $n-1$ hypersurfaces of degree $d$ in $\mathbb{P}^n$.
Assume the following conditions:

\smallskip

\noindent
$(1)$ There is a line bundle $\mathcal{L}$ on $S$ such that the complete linear system $|\mathcal{L}|_{F}|$ of the restriction of $\mathcal{L}$ to the general fiber $F$ defines the embedding $F\subset \mathbb{P}^n$ as complete intersection.

\smallskip

\noindent
$(2)$ Let $X$ be the image of the rational map $\phi_{|\mathcal{L}|/B}\colon S\dasharrow \mathbb{P}_B(f_{*}\mathcal{L})$.
Then any fiber $X_t=X\cap \mathbb{P}(f_{*}\mathcal{L}\otimes \mathbb{C}(t))$ over $t\in B$ is a complete intersection of $n-1$ hypersurfaces of degree $d$ in $\mathbb{P}(f_{*}\mathcal{L}\otimes \mathbb{C}(t))=\mathbb{P}^n$.

\smallskip

\noindent
Then we have
$$
K_f^2\ge \lambda_{n,d}\chi_f,
$$
where the rational number $\lambda_{n,d}$ is defined by
$$
\lambda_{n,d}=\frac{24(n-1)d-24(n+1)}{(3n-2)d-(3n+2)}.
$$

\end{thm}
Theorem~\ref{signthm} is obtained as an application of Theorem~\ref{mainthm}.

This paper is organized as follows.
In \S1, we compute some invariants of normal complete intersection surfaces in a projective bundle on a non-singular projective curve.
In particular, there are many examples of fibered surfaces whose general fiber is a complete intersection of $n-1$ hypersurfaces of degree $d$ in $\mathbb{P}^n$ with slope $\lambda_{n,d}$.
In \S2, we give a proof of Theorem~\ref{mainthm} and state a conjecture about complete intersection curve fibrations.
In \S3, we show Theorem~\ref{signthm} as an application of Theorem~\ref{mainthm}.
The proof of Theorem~\ref{signthm} is similar to the argument in \cite{Eno} which gives a Durfee-type inequality for hypersurface surface singularities.

\begin{acknowledgement}
I would like to express special thanks to Prof.\ Kazuhiro Konno and Prof.\ Tadashi Ashikaga for a lot of discussions and supports.
I am grateful to Prof.\ Tomohiro Okuma, Prof.\ Masataka Tomari and Prof.\ Tadashi Tomaru for variable discussions and comments for singularities.
I also thank Doctors Hiroto Akaike, Sho Ejiri and Kohei Kikuta for discussions.
The research is supported by JSPS KAKENHI No.\ 16J00889.
\end{acknowledgement}

\section{Complete intersection surfaces in a projective bundle over a curve}

Let $B$ be a non-singular projective curve of genus $b$, $\mathcal{E}$ a locally free sheaf of rank $n+1$ on $B$ and $\pi_{W}\colon W:=\mathbb{P}_{B}(\mathcal{E}) \to B$ the associated projective bundle on $B$.
Let $\mathcal{H}_i \in |\mathcal{O}_{W}(d)\otimes \pi_{W}^{*}\mathfrak{a_i}|$, $i=1\ldots,n-1$ be hypersurfaces of relative degree $d$ such that $X:=\mathcal{H}_1\cap \cdots \cap \mathcal{H}_{n-1}$ is a normal irreducible surface and $\overline{f}:=\pi_{W}|_{X}\colon X\to B$ the projection.
We consider the following two invariants:
\begin{align*}
K_{\overline{f}}^2&=(K_X-\overline{f}^{*}K_B)^2=K_X^2-8(g-1)(b-1) \\
\chi_{\overline{f}}&=\chi(\mathcal{O}_X)-(g-1)(b-1),
\end{align*}
where $g$ is the genus of a general fiber $F$ of $\overline{f}$.
These invariants can be computed easily as follows.

\begin{lem} \label{exalem}

Let $a_i$ be the degree of $\mathfrak{a}_i$.
Then we have
$$
K_{\overline{f}}^2=((n-1)d-(n+1))(d-1)d^{n-2}\left((n-1)d\mathrm{deg}\mathcal{E}+(n+1)\sum_{i=1}^{n-1}a_i\right)
$$
and
$$
\chi_{\overline{f}}=\frac{1}{24}((3n-2)d-(3n+2))(d-1)d^{n-2}\left((n-1)d\mathrm{deg}\mathcal{E}+(n+1)\sum_{i=1}^{n-1}a_i\right).
$$
In particular, $K_{\overline{f}}^2=\lambda_{n,d}\chi_{\overline{f}}$ holds.
\end{lem}

\begin{proof}
Put $e':=(n-1)d-n-1$.
Let $T_W$ and $\Gamma_W$ respectively denote the numerical equivalence classes of the tautological line bundle $\mathcal{O}_{W}(1)$  and the fiber of $\pi_{W}$.
Since
$$
K_{W}\sim -(n+1)\mathcal{O}_W(1)+\pi_{W}^{*}(K_B+\wedge^{n+1}\mathcal{E})
$$
and
$$
K_X=(K_W+\mathcal{H}_1+\cdots+\mathcal{H}_{n-1})|_{X},
$$
we have 
\begin{align*}
K_X^2&=(K_W+\mathcal{H}_1+\cdots+\mathcal{H}_{n-1})^2\mathcal{H}_1\cdots \mathcal{H}_{n-1} \\
&=\left(e'T_W+(\sum_{i=1}^{n-1}a_i+2b-2+\mathrm{deg}\mathcal{E})\Gamma_W\right)^2(dT_W+a_1\Gamma_W)\cdots (dT_W+a_{n-1}\Gamma_W) \\
&=e'd^{n-1}(e'+2)\mathrm{deg}\mathcal{E}+e'd^{n-2}(e'+2d)\sum_{i=1}^{n-1}a_i+4e'd^{n-1}(b-1),
\end{align*}
where the last equality follows from $T_{W}^{n+1}=\mathrm{deg}\mathcal{E}$, $T_{W}^{n}\Gamma_{W}=1$ and $\Gamma_{W}^2=0$.
Thus we have
$$
K_{\overline{f}}^2=K_X^2-4e'd^{n-1}(b-1)=e'd^{n-1}(e'+2)\mathrm{deg}\mathcal{E}+e'd^{n-2}(e'+2d)\sum_{i=1}^{n-1}a_i.
$$

Next, we compute $\chi_{\overline{f}}$.
We put
\begin{align*}
\chi(d,a)&:=\chi(\mathcal{O}_{W}(-d)\otimes \pi_{W}^{*}\mathcal{O}_{B}(-\mathfrak{a})) \\
&=(-1)^{n+1}\left(\binom{d-1}{n}(b-1)+\binom{d}{n+1}\mathrm{deg}\mathcal{E}+\binom{d-1}{n}a\right),
\end{align*}
where $a$ is the degree of a divisor $\mathfrak{a}$ on $B$.
From the exact sequence
$$
0\to \mathcal{O}_{\mathcal{H}_1\cap \cdots \cap\mathcal{H}_{i-1}}(-\sum_{j=1}^{l}\mathcal{H}_{n_j}-\mathcal{H}_i)\to \mathcal{O}_{\mathcal{H}_1\cap \cdots \cap\mathcal{H}_{i-1}}(-\sum_{j=1}^{l}\mathcal{H}_{n_j}) \to \mathcal{O}_{\mathcal{H}_1\cap \cdots \cap\mathcal{H}_{i}}(-\sum_{j=1}^{l}\mathcal{H}_{n_j}) \to 0,
$$
we can show by the induction on $i$ that
$$
\chi(\mathcal{O}_{\mathcal{H}_1\cap\cdots \cap\mathcal{H}_i}(-\sum_{j=1}^{l}\mathcal{H}_{n_j}))=\sum_{k=0}^{i}(-1)^{k}\sum_{1\le j_1< \cdots < j_k\le i}\chi((k+l)d,a_{j_1}+\cdots+a_{j_k}+a_{n_1}+\cdots+a_{n_l}).
$$
Putting $i=n-1$ and $l=0$, we have
\begin{align*}
\chi(\mathcal{O}_X)&=\sum_{k=0}^{n-1}(-1)^k\sum_{1\le j_1<\cdots<j_k\le n-1}\chi(kd,a_{j_1}+\cdots+a_{j_k}) \\
&=A_0(b-1)+A_1\mathrm{deg}\mathcal{E}+A_2\sum_{i=1}^{n-1}a_i,
\end{align*}
where 
\begin{align*}
A_0:=&\sum_{k=0}^{n-1}(-1)^{n+k+1}\binom{n-1}{k}\binom{kd-1}{n}, \\
A_1:=&\sum_{k=0}^{n-1}(-1)^{n+k+1}\binom{n-1}{k}\binom{kd}{n+1}
\end{align*}
and
$$
A_2:=\sum_{k=0}^{n-1}(-1)^{n+k+1}\binom{n-2}{k-1}\binom{kd-1}{n}.
$$
For integers $m\ge 1$ and $l\ge 0$, put
$$
\sigma_{m,l}:=\sum_{k=0}^{m}(-1)^k\binom{m}{k}k^l.
$$
Thus we can see that $\sigma_{m,l}=0$ for $m>l$ and
\begin{align*}
\sigma_{m,m}&=(-1)^m m!, \\
\sigma_{m,m+1}&=\frac{(-1)^m}{2}m(m+1) m!, \\
\sigma_{m,m+2}&=\frac{(-1)^m}{24}m(m+1)(m+2)(3m+1)m!
\end{align*}
from the properties $\sigma_{m,l}=m(\sigma_{m,l-1}-\sigma_{m-1,l-1})$ and $\sigma_{m,0}=0$.
Thus we can compute $A_0$, $A_1$ and $A_2$ as follows.
\begin{align*}
A_0&=\frac{(-1)^{n+1}}{n!}d^{n-1}\left(d\sigma_{n-1,n}-\frac{n(n+1)}{2}\sigma_{n-1,n-1}\right) \\
&=\frac{d^{n-1}e'}{2}. \\
A_1&=\frac{(-1)^{n+1}}{(n+1)!}d^{n-1}\left(d^2\sigma_{n-1,n+1}-\frac{n(n+1)}{2}d\sigma_{n-1,n}+\frac{n(n+1)(n-1)(3n+2)}{24}\sigma_{n-1,n-1}\right) \\
&=\frac{(3n-2)d-(3n+2)}{24}d^{n-1}(d-1)(n-1). \\
A_2&=\frac{(-1)^{n}}{n!}d^{n-2}\left(d^2\sigma_{n-2,n}+\frac{nd}{2}(2d-n(n+1))\sigma_{n-2,n-1}\right. \\
&\ \ \ + \left.\frac{n(n-1)}{24}(12d^2-12(n+1)d+(n+1)(3n+2))\sigma_{n-2,n-2}\right) \\
&=\frac{(3n-2)d-(3n+2)}{24}d^{n-2}(d-1)(n+1).
\end{align*}
Thus we have
$$
\chi_{\overline{f}}=\chi(\mathcal{O}_X)-\frac{d^{n-1}e'}{2}(b-1)=\lambda_{n,d}^{-1}K_{\overline{f}}^2.
$$
\end{proof}

\begin{rem}
Taking $\mathcal{H}_i$'s sufficiently general, the surface $X$ is non-singular by Bertini's theorem.
Thus the fibration $f=\overline{f}\colon X\to B$ is a fibered surface whose general fiber is a complete intersection of $n-1$ hypersurfaces of degree $d$ in $\mathbb{P}^n$ with slope $\lambda_{n,d}$.
\end{rem}

\section{Slope inequality for complete intersection curve fibrations}

In this section, we prove the following theorem:

\begin{thm} \label{Mainthm}
Let $f\colon S\to B$ be a relatively minimal fibration whose general fiber $F$ is a complete intersection of $n-1$ hypersurfaces of degree $d$ in $\mathbb{P}^n$.
Assume the following conditions:

\smallskip

\noindent
$(1)$ There is a line bundle $\mathcal{L}$ on $S$ such that the complete linear system $|\mathcal{L}|_{F}|$ of the restriction of $\mathcal{L}$ to the general fiber $F$ defines the embedding $F\subset \mathbb{P}^n$ as complete intersection.

\smallskip

\noindent
$(2)$ Let $X$ be the image of the rational map $\phi_{|\mathcal{L}|/B}\colon S\dasharrow \mathbb{P}_B(f_{*}\mathcal{L})$.
Then any fiber $X_t=X\cap \mathbb{P}(f_{*}\mathcal{L}\otimes \mathbb{C}(t))$ over $t\in B$ is a complete intersection of $n-1$ hypersurfaces of degree $d$ in $\mathbb{P}(f_{*}\mathcal{L}\otimes \mathbb{C}(t))=\mathbb{P}^n$.

\smallskip

\noindent
Then we have
$$
K_f^2\ge \lambda_{n,d}\chi_f,
$$
where the rational number $\lambda_{n,d}$ is defined by
$$
\lambda_{n,d}=\frac{24(n-1)d-24(n+1)}{(3n-2)d-(3n+2)}.
$$

\end{thm}

Let $f\colon S\to B$ be a relatively minimal fibration of genus $g$ whose general fiber is a complete intersection $F=H_1\cap\cdots \cap H_{n-1}\subset \mathbb{P}^n$, $H_i\in |\mathcal{O}_{\mathbb{P}^n}(d)|$ and $\mathcal{L}$ a line bundle on $S$ satisfying the conditions $(1)$, $(2)$ in Theorem~\ref{Mainthm}.
Put $e':=(n-1)d-n-1$.
Then we have $e'\mathcal{L}|_F\sim K_F$ from the adjunction formula.
Thus there exists an $f$-vertical divisor $D$ on $S$ such that $e'\mathcal{L}+D\sim K_f$.
Let $m$ be a positive integer and $e:=me'$.
Then there is a natural injection $f_{*}\mathcal{L}^{\otimes e}\to f_{*}\omega_f^{\otimes m}(\mathfrak{b})$ whose cokernel is a torsion sheaf by taking a sufficiently effective divisor $\mathfrak{b}$ on $B$.
Let $\mathcal{K}$ be the kernel of the natural homomorphism $\mathrm{Sym}^{e}f_{*}\mathcal{L}\to f_{*}\mathcal{L}^{\otimes e}$.
Then we have an exact sequence
\begin{equation} \label{Kex}
0\to \mathcal{K}\to \mathrm{Sym}^{e}f_{*}\mathcal{L}\to f_{*}\omega_f^{\otimes m}(-\mathfrak{c})\to \mathcal{T}\to 0,
\end{equation}
where the divisor $\mathfrak{c}$ is determined as follows.
Let $\mathfrak{b}'$ be a maximal divisor on $B$ satisfying 
$$
\mathrm{Im}(\mathrm{Sym}^{e}f_{*}\mathcal{L}\to f_{*}\mathcal{L}^{\otimes e}\to f_{*}\omega_f^{\otimes m}(\mathfrak{b}))\subset f_{*}\omega_f^{\otimes m}(\mathfrak{b}-\mathfrak{b}').
$$
Then we put $\mathfrak{c}:=\mathfrak{b}'-\mathfrak{b}$ (not necessarily effective).
Thus the elementary transformation 
$$
P':=\mathbb{P}_B(\mathrm{Sym}^{e}f_{*}\mathcal{L}/\mathcal{K})\xleftarrow{\tau'} \widetilde{P}\xrightarrow{\tau} P:=\mathbb{P}_{B}(f_{*}\omega_f^{\otimes m})
$$
occurs with
$$
\tau^{*}\mathcal{O}_{P}(1)-\tau'^{*}\mathcal{O}_{P'}(1)=\widetilde{\pi}^{*}\mathfrak{c}+E_{\tau},
$$
where $\widetilde{\pi}\colon \widetilde{P}\to B$ is the projection and $E_{\tau}$ is an effective $\tau$-exceptional divisor.
Put $\mathbf{P}:=\mathbb{P}_B(\mathrm{Sym}^{e}f_{*}\mathcal{L})$ and $W:=\mathbb{P}_{B}(f_{*}\mathcal{L})$.
Then there are the natural embedding $P'\subset \mathbf{P}$ as a relative plane and the relative Veronese embedding $W\subset \mathbf{P}$ of degree $e$.
Clearly, two rational maps $S\dasharrow W\subset \mathbf{P}$ and $S\dasharrow P\dasharrow P'\subset \mathbf{P}$ are coincide.
Let $\rho\colon \widetilde{S}\to S$ be a resolution of indeterminacy of two rational maps $S\dasharrow \mathbf{P}$ and $S\dasharrow \widetilde{P}$ and $\mathbf{\Phi}\colon \widetilde{S}\to \mathbf{P}$ and $\widetilde{\Phi}\colon \widetilde{S}\to \widetilde{P}$ the induced morphisms.
Put $T:=\tau^{*}\mathcal{O}_P(1)$, $T':=\tau'^{*}\mathcal{O}_{P'}(1)$, $\mathbf{T}:=\mathcal{O}_{\mathbf{P}}(1)$ and $T_{W}:=\mathcal{O}_{W}(1)$.
Then we have $\mathbf{T}|_{P'}=T'$ and $\mathbf{T}|_{W}=eT_{W}$.
Let $F$, $\widetilde{F}$, $\Gamma$, $\Gamma'$, $\widetilde{\Gamma}$, $\mathbf{\Gamma}$ and $\Gamma_{W}$ respectively be the numerical equivalence classes of the fiber of the natural projections $f\colon S\to B$, $\widetilde{f}\colon \widetilde{S}\to B$, $\pi\colon P\to B$, $\pi'\colon P'\to B$, $\widetilde{\pi}\colon \widetilde{P}\to B$, $\mathbf{\Pi}\colon \mathbf{P}\to B$ and $\pi_{W}\colon W\to B$.
Then we can write $T-T'\equiv c\widetilde{\Gamma}+E_{\tau}$, where the symbol $\equiv$ means the numerical equivalence and $c:=\mathrm{deg}\mathfrak{c}$.
Put $M:=\widetilde{\Phi}^{*}T$ and $M':=\widetilde{\Phi}^{*}T'=\mathbf{\Phi}^{*}\mathbf{T}$.
Then we have 
\begin{equation} \label{KfMeq}
m\rho^{*}K_f=M+Z
\end{equation}
for some effective $\widetilde{f}$-vertical divisor $Z$.
We can also write $K_{\widetilde{f}}=\rho^{*}K_f+E$ for some effective $\rho$-exceptional divisor $E$.
Put $X:=\mathbf{\Phi}(\widetilde{S})$, which is nothing but the image of $\phi_{|\mathcal{L}|/B}\colon S\dasharrow W$.
Then we can write as cycles in $\mathbf{P}$,
$$
W\equiv e^n\mathbf{T}^{N-n}+\alpha\mathbf{T}^{N-n-1}\mathbf{\Gamma}
$$
and
$$
X\equiv d^{n-1}e\mathbf{T}^{N-1}+\beta\mathbf{T}^{N-2}\mathbf{\Gamma}
$$
for some integers $\alpha$ and $\beta$,
where $N:=\binom{e+n}{n}-1$.
Thus we can compute
\begin{equation} \label{alphaeq}
\left(\mathbf{T}|_{W}\right)^{n+1}=\mathbf{T}^{n+1}W=e^{n}\mathrm{deg}(\mathrm{Sym}^{e}f_{*}\mathcal{L})+\alpha
\end{equation}
and
\begin{equation} \label{betaeq}
M'^2=\mathbf{T}^{2}X=d^{n-1}e\mathrm{deg}(\mathrm{Sym}^{e}f_{*}\mathcal{L})+\beta.
\end{equation}
On the other hand, we have
\begin{equation} \label{tautoeq}
\left(\mathbf{T}|_{W}\right)^{n+1}=\left(eT_{W}\right)^{n+1}=e^{n+1}\mathrm{deg}f_{*}\mathcal{L}.
\end{equation}
We can regard $X$ as a cycle in $W$, which is denoted by $[X]_{W}$.
Then we can write
$$
[X]_{W}\equiv d^{n-1}T_{W}^{n-1}+\beta_W T_{W}^{n-2}\Gamma_{W}=\left(\frac{d^{n-1}}{e^{n-1}}\mathbf{T}^{n-1}+\frac{\beta_W}{e^{n-2}}\mathbf{T}^{n-2}\mathbf{\Gamma}\right)|_{W}
$$
for some integer $\beta_W$.
The numerical equivalence
$$
\left(\frac{d^{n-1}}{e^{n-1}}\mathbf{T}^{n-1}+\frac{\beta_W}{e^{n-2}}\mathbf{T}^{n-2}\mathbf{\Gamma}\right)W\equiv X
$$
induces 
\begin{equation} \label{alphabetaeq}
\beta=e^{2}\beta_W+\frac{d^{n-1}}{e^{n-1}}\alpha.
\end{equation}
Combining \eqref{alphabetaeq} with \eqref{alphaeq}, \eqref{betaeq} and \eqref{tautoeq}, we have
\begin{align}
\beta_W&=\frac{1}{e^{2}}\beta-\frac{d^{n-1}}{e^{n+1}}\alpha \nonumber \\
&=\frac{1}{e^{2}}M'^2-\frac{d^{n-1}}{e^{n+1}}\left(\mathbf{T}|_{W}\right)^{n+1} \nonumber \\
&=\frac{1}{e^2}M'^{2}-d^{n-1}\mathrm{deg}f_{*}\mathcal{L}. \label{beta''eq}
\end{align}

From the condition $(2)$ in Theorem~\ref{Mainthm}, we can see that $X$ is the zeros of a section of a vector bundle $V$ of rank $n-1$ on $W$:

\begin{lem} \label{zeroseq}
There exist a vector bundle $V$ on $W$ of rank $n-1$ and a section $s\in H^{0}(W,V)$ such that $X$ is the zeros of $s$ and $V|_{\pi_{W}^{-1}(t)}\simeq \mathcal{O}_{\pi_{W}^{-1}(t)}(d)^{\oplus n-1}$ for any $t\in B$.
\end{lem}

\begin{proof}
Let $\mathcal{K}_d$ be the kernel of the natural homomorphism 
$$
\mathrm{Sym}^df_{*}\mathcal{L}=\pi_{W*}\mathcal{O}_{W}(d)\to \pi_{W*}\mathcal{O}_{X}(d)
$$
and regard it as a relative linear system of $\pi_W$.
Thus by the condition $(2)$ in Theorem~\ref{Mainthm}, $X$ coincides with the relative base locus $\mathrm{Bs}(\mathcal{K}_d)$, that is, $X$ is correspond to the ideal sheaf defined by the image of the homomorphism $ev\colon \pi_{W}^{*}\mathcal{K}_d\otimes \mathcal{O}_{W}(-d)\to \mathcal{O}_W$.
Hence $X$ is the zeros of the global section $s$ of $\mathcal{O}_W(d)\otimes \pi_W^{*}\mathcal{K}_d^{*}$ corresponding to the dual of $ev$. 
The vector bundle $V:=\mathcal{O}_W(d)\otimes \pi_W^{*}\mathcal{K}_d^{*}$ is a desired one.
\end{proof}
By the splitting principle, we may assume that $c(V)=\prod_{i=1}^{n-1}(1+\rho_i)$, $\rho_i\equiv dT_W+a_i\Gamma_W$, where $c(V)=\sum_{i}c_i(V)$ is the total Chern class of $V$ and then $\mathrm{deg}\mathcal{K}_d=-\sum_{i=1}^{n-1}a_i$.
Since $[X]_W=c_{n-1}(V)=\prod_{i=1}^{n-1}\rho_i$, we have
\begin{equation} \label{beta''aieq}
\beta_W=d^{n-2}\sum_{i=1}^{n-1}a_i=-d^{n-2}\mathrm{deg}\mathcal{K}_d.
\end{equation}

\begin{rem}
For example, if $X$ is a complete intersection $X=\mathcal{H}_1\cap \cdots \cap \mathcal{H}_{n-1}$, $\mathcal{H}_i\in |\mathcal{O}_W(d)\otimes \pi_W^{*}\mathfrak{a}_i|$, then the vector bundle $V$ is the direct sum $\bigoplus_{i=1}^{n-1}\left(\mathcal{O}_W(d)\otimes \pi_W^{*}\mathfrak{a}_i\right)$ and $a_i=\mathrm{deg}\mathfrak{a}_i$.
\end{rem}

Since $M-M'\equiv c\widetilde{F}+\widetilde{\Phi}^{*}E_{\tau}$ and $d^{n-1}e'=2g-2$, we get
$$
M'^{2}=M^{2}-2d^{n-1}ec+(\widetilde{\Phi}^{*}E_{\tau})^{2}.
$$
We can compute $(\widetilde{\Phi}^{*}E_{\tau})^{2}$ as follows.
\begin{align*}
(\widetilde{\Phi}^{*}E_{\tau})^{2}&=(M-M'-c\widetilde{F})\widetilde{\Phi}^{*}E_{\tau} \\
&=-M'\widetilde{\Phi}^{*}E_{\tau} \\
&=-T'E_{\tau}\widetilde{\Phi}(\widetilde{S}) \\
&=-d^{n-1}eT'^{n'}E_{\tau} \\
&=-d^{n-1}eT'^{n'}(T-T'-c\widetilde{\Gamma}) \\
&=-d^{n-1}e((T-c\widetilde{\Gamma}-E_{\tau})^{n'}T-T'^{n'+1}-c) \\
&=-d^{n-1}e((T^{n'}-n'cT^{n'-1}\widetilde{\Gamma})T-T'^{n'+1}-c) \\
&=-d^{n-1}e(T^{n'+1}-T'^{n'+1}-(n'+1)c) \\
&=-d^{n-1}e\ell, \\
\end{align*}
where $n':=(2m-1)(g-1)-1$ and $\ell:=\mathrm{length}(\mathcal{T})$.
Thus we have
\begin{equation} \label{MM'eq}
M'^2=M^2-d^{n-1}e(2c+\ell).
\end{equation}
The complete linear system $|M'+\widetilde{f}^{*}\mathfrak{a}|$ is free from base points for a sufficiently ample divisor $\mathfrak{a}$ on $B$.
Thus by Bertini's Theorem, we can take a smooth general member $C\in |M'+\widetilde{f}^{*}\mathfrak{a}|$.
Put $C':=\mathbf{\Phi}(C)$.
Then we can compute the arithmetic genera of $C$ and $C'$ independently
and the difference $p_a(C')-g(C)$ is as follows.

\begin{lem} \label{CC'lem}
We have
\begin{align*}
0\le 2p_a(C')-2g(C)&=\frac{d}{e}M^2-d^{n-1}e(d-1)\mathrm{deg}f_{*}\mathcal{L} \nonumber \\
&\ -d^{n-1}(e'+2d)c-d^{n-1}(e'+d)\ell-\left(\frac{1}{m}Z+E\right)C.
\end{align*}
\end{lem}

\begin{proof}
From the adjunction formula $K_C=(K_{\widetilde{S}}+C)|_{C}$, we have
\begin{align}
2g(C)-2&=(K_{\widetilde{S}}+C)C \nonumber \\
&=(\rho^{*}K_f+E+(2b-2)\widetilde{F}+C)C \nonumber \\
&=(\rho^{*}K_f+E+(2b-2)\widetilde{F}+M'+a\widetilde{F})(M'+a\widetilde{F}) \nonumber \\
&=\left(\left(1+\frac{1}{m}\right)M+\frac{1}{m}Z+E+(2b-2+a-c)\widetilde{F}-\widetilde{\Phi}^{*}E_{\tau}\right) \cdot  \nonumber \\
& \ \ \ \ \ \left(M+(a-c)\widetilde{F}-\widetilde{\Phi}^{*}E_{\tau}\right) \nonumber \\
&=\left(1+\frac{1}{m}\right)M^{2}+\left(\frac{1}{m}Z+E\right)C \nonumber \\
&\ +d^{n-1}e\left(2b-2+\left(2+\frac{1}{m}\right)(a-c)-\ell\right), \label{Ceq} 
\end{align}
where $b:=g(B)$ and $a:=\mathrm{deg}\mathfrak{a}$.

Next, we compute the arithmetic genus $p_a(C')$.
Since $C\to C'$ is the normalization of $C'$, the exact sequence
\begin{equation*}
0\to \mathcal{O}_{C'}\to \varphi_{*}\mathcal{O}_C\to \mathcal{Q}\to 0
\end{equation*}
holds on $W$, where we put $\varphi:=\mathbf{\Phi}|_C\colon C\to C'\subset W$ and the cokernel $\mathcal{Q}$ is a torsion sheaf on $W$ which satisfies $ch(\mathcal{Q})=ch_{n+1}(\mathcal{Q})=p_a(C')-g(C)$.
Then we have
\begin{equation} \label{normcheq}
ch(\varphi_{*}\mathcal{O}_C)=ch(\mathcal{O}_{C'})+p_a(C')-g(C).
\end{equation}
Applying the Grothendiek Riemann-Roch theorem to $\varphi\colon C\to W$ and $\varphi_{*}\mathcal{O}_C$, we have
$$
ch(\varphi_{*}\mathcal{O}_C)td(\mathcal{T}_W)=\varphi_{*}(ch(\mathcal{O}_C)td(\mathcal{T}_C)),
$$
where $\mathcal{T}_Y$ is the tangent sheaf of a variety $Y$ and note that $R^{k}\varphi_{*}\mathcal{O}_C=0$ for $k\ge 1$.
By computing this, we have $ch_n(\varphi_{*}\mathcal{O}_C)=[C']_{W}$ and 
$$
ch_{n+1}(\varphi_{*}\mathcal{O}_C)=1-g(C)+\frac{1}{2}K_WC'.
$$
Combining this with \eqref{normcheq}, we have
\begin{equation} \label{chc'eq}
ch_{n+1}(\mathcal{O}_{C'})=1-p_a(C)+\frac{1}{2}K_WC'.
\end{equation}
On the other hand, from Lemma~\ref{zeroseq}, $C'$ is the zeros of some global section $(s,s')$ of the vector bundle $V\oplus L$, where $L:=\mathcal{O}_W(e)\otimes \pi_W^{*}\mathfrak{a}$ and hence $\mathcal{O}_{C'}$ has the Koszul resolution
$$
0\to \wedge^{n}(V\oplus L)^{*}\to \wedge^{n-1}(V\oplus L)^{*}\to \cdots \to \wedge^{2}(V\oplus L)^{*}\to (V\oplus L)^{*}\xrightarrow{(s,s')^{*}} \mathcal{O}_{W}\to \mathcal{O}_{C'}\to 0.
$$
It follows that
\begin{align*}
ch(\mathcal{O}_{C'})&=\sum_{k=0}^{n}(-1)^{k}\sum_{1\le i_1<\cdots<i_k\le n}e^{-(\rho_{i_1}+\cdots+\rho_{i_k})} \\
&=\sum_{k=0}^{n}(-1)^{k}\sum_{1\le i_1<\cdots<i_k\le n}\sum_{j=0}^{n+1}(-1)^{j}\frac{(\rho_{i_1}+\cdots+\rho_{i_k})^{j}}{j!} \\
&=\sum_{j=0}^{n+1}\left(\sum_{k=0}^{n}(-1)^{k+j}\sum_{1\le i_1<\cdots<i_k\le n}\frac{(\rho_{i_1}+\cdots+\rho_{i_k})^{j}}{j!}\right) \\
&=\left(1-\frac{1}{2}\sum_{i=1}^{n}\rho_i\right)\prod_{i=1}^{n}\rho_i,
\end{align*}
where $\rho_n:=c_1(L)\equiv eT_W+a\Gamma_W$.
Thus we have
\begin{equation}\label{chc'eq2}
ch_{n+1}(\mathcal{O}_{C'})=-\frac{1}{2}\left(((n-1)d+e)T_W+\left(\sum_{i=1}^{n-1}a_i+a\right)\Gamma_W\right)C'.
\end{equation}
Substituting \eqref{chc'eq2} into \eqref{chc'eq}, we have

\begin{align}
2p_a(C')-2
&=\left(\left(1+\frac{1}{m}\right)eT_{W}+\left(2b-2+\mathrm{deg}f_{*}\mathcal{L}+\sum_{i=1}^{n-1}a_i+a\right)\Gamma_{W}\right)C' \nonumber \\
&=d^{n-1}e\left(\left(1+\frac{1}{m}\right)e+1\right)\mathrm{deg}f_{*}\mathcal{L}+e\left(\left(1+\frac{1}{m}\right)e+d\right)\beta_W \nonumber \\
&\ +d^{n-1}e\left(2b-2+\left(2+\frac{1}{m}\right)a\right) \nonumber \\
&=\frac{1}{e}\left(\left(1+\frac{1}{m}\right)e+d\right)(M^2-d^{n-1}e(2c+\ell))-d^{n-1}e(d-1)\mathrm{deg}f_{*}\mathcal{L} \nonumber \\
& \ +d^{n-1}e\left(2b-2+\left(2+\frac{1}{m}\right)a\right). \label{C'eq} 
\end{align}

Subtracting \eqref{Ceq} from \eqref{C'eq}, we have the desired equation.
\end{proof}

\begin{rem}
If $X$ is a complete intersection $X=\mathcal{H}_1\cap \cdots \cap \mathcal{H}_{n-1}$, we can calculate $p_a(C')$ more easily by the adjunction formula.
\end{rem}

Next, we show the following equality among $K_f^2$, $\chi_f$, $c$ and $\ell$.

\begin{lem} \label{f*Llem}
We have
\begin{align*}
(u-dv)\mathrm{deg}f_{*}\mathcal{L}=&\left(\frac{1}{2}m(m-1)-\frac{m^2v}{d^{n-2}e^2}\right)K_f^2+\chi_f \nonumber \\
&+\left(\frac{2dv}{e}-\left(1-\frac{1}{2m}\right)d^{n-1}e\right)c+\left(\frac{dv}{e}-1\right)\ell
\end{align*}
for $m\gg 0$, where
\begin{align*}
u:=&\frac{d^{n-1}}{2}\left(e^2-((n-1)d-n)e+2(n-1)r\right), \\
v:=&\frac{d^{n-2}}{2}\left(e^2-(nd-(n+1))e+2(n+1)r\right)
\end{align*}
and 
$$
r:=\frac{1}{24}(d-1)((3n-2)d-(3n+2)).
$$
\end{lem}

\begin{proof}
We can see that $\mathcal{K}\simeq \pi_{W*}\mathcal{I}_{X/W}(e)$, where $\mathcal{I}_{X/W}$ is the ideal sheaf of $X$ in $W$, since the natural homomorphism
$$
\pi_{W*}\mathcal{O}_X(e)\to \widetilde{f}_{*}\mathbf{\Phi}^{*}\mathcal{O}_W(e)\to \widetilde{f}_{*}\rho^{*}\mathcal{L}^{\otimes e}=f_{*}\mathcal{L}^{\otimes e}
$$
is injective and the multiplicative map is decomposed as follows:
$$
\mathrm{Sym}^{e}f_{*}\mathcal{L}=\pi_{W*}\mathcal{O}_W(e)\to \pi_{W*}\mathcal{O}_X(e)\to f_{*}\mathcal{L}^{\otimes e}.
$$
From the Serre vanishing theorem, we have $R^{i}\pi_{W*}\mathcal{I}_{X/W}(e)=R^{i}\pi_{W*}\mathcal{O}_{X}(e)=0$ for $i>0$ and $m\gg 0$
and then the following exact sequence holds:
$$
0\to \pi_{W*}\mathcal{I}_{X/W}(e) \to \pi_{W*}\mathcal{O}_{W}(e) \to \pi_{W*}\mathcal{O}_{X}(e) \to 0.
$$
Thus, from \eqref{Kex}, we have
\begin{equation} \label{OXedegeq}
\mathrm{deg}\pi_{W*}\mathcal{O}_{X}(e)=\mathrm{deg}f_{*}\omega_f^{\otimes m}(-\mathfrak{c})-\ell=\frac{m(m-1)}{2}K_f^2+\chi_f-(2m-1)(g-1)c-\ell.
\end{equation}

On the other hand, from the Grothendieck Riemann-Roch theorem, we have
$$
ch(\pi_{W*}\mathcal{O}_{X}(e))=\pi_{W*}(ch(\mathcal{O}_{X}(e))td(\mathcal{T}_{\pi_{W}})),
$$
where $\mathcal{T}_{\pi_{W}}=\Omega_{W/B}^{*}$ is the relative tangent sheaf of $\pi_W\colon W\to B$.
From the Euler sequence
$$
0\to \Omega_{W/B}(1)\to \pi_{W}^{*}f_{*}\mathcal{L}\to \mathcal{O}_{W}(1)\to 0,
$$
we have $c(\mathcal{T}_{\pi_{W}})=c((\pi_{W}^{*}f_{*}\mathcal{L})^{*}\otimes \mathcal{O}_{W}(1))$.
By the splitting principle, we may assume that $c(f_{*}\mathcal{L})=\prod_{i=1}^{n+1}(1+\eta_i)$ for some $\eta_i\in A^{1}(B)$.
Then 
\begin{align*}
c(\mathcal{T}_{\pi_{W}})&=c((\pi_{W}^{*}f_{*}\mathcal{L})^{*}\otimes \mathcal{O}_{W}(1)) \\
&=\prod_{i=1}^{n+1}(1-\pi_{W}^{*}\eta_{i}+T_{W}) \\
&=1+\left((n+1)T_{W}-\pi_{W}^{*}c_{1}(f_{*}\mathcal{L})\right)+\left(\frac{n(n+1)}{2}T_{W}^{2}-nT_{W}\pi_{W}^{*}c_{1}(f_{*}\mathcal{L})\right)+\cdots.\\
\end{align*}
Hence we get
\begin{equation} \label{tdeq}
td(\mathcal{T}_{\pi_{W}})=1+\frac{1}{2}\nu_1+\frac{1}{12}(\nu_1^{2}+\nu_2)+\cdots,
\end{equation}
where 
$$
\nu_1=(n+1)T_{W}-\pi_{W}^{*}c_{1}(f_{*}\mathcal{L}), \quad \nu_2=\frac{n(n+1)}{2}T_{W}^{2}-nT_{W}\pi_{W}^{*}c_{1}(f_{*}\mathcal{L}).
$$
While we have
\begin{equation} \label{cheeq}
ch(\mathcal{O}_{X}(e))=ch(\mathcal{O}_{X})ch(\mathcal{O}_{W}(e))=ch(\mathcal{O}_{X})\left(1+eT_{W}+\frac{e^2}{2}T_{W}^2+\cdots\right)
\end{equation}
and we can compute $ch(\mathcal{O}_{X})$ as follows.
The subvariety $X$ can be regarded as the zeros of a section $s$ of the vector bundle $V$ from Lemma~\ref{zeroseq}
and hence $\mathcal{O}_X$ has the Koszul resolution
$$
0\to \wedge^{n-1}V^{*}\to \wedge^{n-2}V^{*}\to \cdots \to \wedge^{2}V^{*}\to V^{*}\xrightarrow{s^{*}} \mathcal{O}_{W}\to \mathcal{O}_{X}\to 0.
$$
Thus we have 
\begin{align}
ch(\mathcal{O}_{X})&=\sum_{k=0}^{n-1}(-1)^{k}ch(\wedge^{k}V^{*}) \nonumber \\
&=\sum_{k=0}^{n-1}(-1)^{k}\sum_{1\le i_1<\cdots<i_k\le n-1}e^{-(\rho_{i_1}+\cdots+\rho_{i_k})} \nonumber \\
&=\sum_{k=0}^{n-1}(-1)^{k}\sum_{1\le i_1<\cdots<i_k\le n-1}\sum_{j=0}^{n+1}(-1)^{j}\frac{(\rho_{i_1}+\cdots+\rho_{i_k})^{j}}{j!} \nonumber \\
&=\sum_{j=0}^{n+1}\left(\sum_{k=0}^{n-1}(-1)^{k+j}\sum_{1\le i_1<\cdots<i_k\le n-1}\frac{(\rho_{i_1}+\cdots+\rho_{i_k})^{j}}{j!}\right) \nonumber \\
&=\left(1-\frac{1}{2}\sum_{i=1}^{n-1}\rho_i+\frac{1}{4}\sum_{i<i'}\rho_{i}\rho_{i'}+\frac{1}{6}\sum_{i=1}^{n-1}\rho_{i}^{2}\right)\prod_{i=1}^{n-1}\rho_i. \label{chXeq}
\end{align}
From \eqref{tdeq}, \eqref{cheeq} and \eqref{chXeq}, we can compute $\mathrm{deg}(\pi_{W*}\mathcal{O}_{X}(e))$ as follows.
\begin{align}
\mathrm{deg}(\pi_{W*}\mathcal{O}_{X}(e))&=\mathrm{deg}\pi_{W*}((ch(\mathcal{O}_{X}(e))td(\mathcal{T}_{\pi_{W}}))_{n+1}) \nonumber \\
&=u\mathrm{deg}f_{*}\mathcal{L}+v\sum_{i=1}^{n-1}a_i \nonumber \\
&=(u-dv)\mathrm{deg}f_{*}\mathcal{L}+\frac{v}{d^{n-2}e^2}M^2-\frac{dv}{e}(2c+\ell). \label{OXedegeq2}
\end{align}
where we use \eqref{beta''eq}, \eqref{beta''aieq} and \eqref{MM'eq}.
Combining \eqref{OXedegeq2} with \eqref{KfMeq}, \eqref{OXedegeq} and $Z=0$ for $m\gg 0$ from the base point free theorem, we have
\begin{align*}
&(u-dv)\mathrm{deg}f_{*}\mathcal{L}+\frac{v}{d^{n-2}e^2}m^2K_f^2-\frac{dv}{e}(2c+\ell) \\
&=\frac{1}{2}m(m-1)K_f^2+\chi_f-\left(1-\frac{1}{2m}\right)d^{n-1}ec-\ell,
\end{align*}
which implies the desired equation.
\end{proof}

\begin{prfofthm1}
Eliminating the term $\mathrm{deg}f_{*}\mathcal{L}$ in Lemma~\ref{CC'lem} by using Lemma~\ref{f*Llem},
we have
$$
K_f^2=\lambda_{n,d}\chi_f+p_1(p_a(C')-g(C))+p_2\ell+p_3EC,
$$
where 
$$
\lambda_{n,d}:=\frac{(d-1)e'}{r}=\frac{24((n-1)d-(n+1))}{(3n-2)d-(3n+2)}
$$
and
the coefficients $p_i$ can be computed and become positive rational numbers.
In particular, we have $K_f^2\ge \lambda_{n,d}\chi_f$.
\end{prfofthm1}





We close this section by stating a conjecture for complete intersection curve fibrations.
A fibered surface $f\colon S\to B$ is said to be a {\em complete intersection curve fibrations of type $(n,d)$} if the general fiber of $f$ is a complete intersection of $n-1$ hypersurfaces of degree $d$ in $\mathbb{P}^n$.

\begin{conj} \label{slopeeqconj}
Let $\mathcal{A}_{n,d}$ be the set of holomorphically equivalence classes of fiber germs of relatively minimal complete intersection curve fibrations of type $(n,d)$.
Then the Horikawa index $\mathrm{Ind}_{n,d}\colon \mathcal{A}_{n,d}\to \mathbb{Q}_{\ge 0}$ is defined such that
for any relatively minimal complete intersection curve fibration $f\colon S\to B$ of type $(n,d)$, the slope equality
$$
K_f^2=\lambda_{n,d}\chi_f+\sum_{p\in B}\mathrm{Ind}_{n,d}(f^{-1}(p))
$$
holds.
\end{conj}

Note that Conjecture~\ref{slopeeqconj} holds for $n=2$ in \cite{Eno}.

\section{The negativity of the signature of the Milnor fiber}

In this section, we prove the following theorem by applying Theorem~\ref{Mainthm}.
\begin{thm} \label{Signthm}
Let $(X,o)$ be a germ of an isolated complete intersection surface singularity with embedding dimension $n$.
Let $\pi\colon \widetilde{X}\to X$ be the minimal resolution of $(X,o)$ and $p_g:=\mathrm{dim}_{\mathbb{C}}R^{1}\pi_{*}\mathcal{O}_{\widetilde{X}}$ the geometric genus of $(X,o)$.
Then we have
$$
\sigma\le -\frac{8}{3n-5}p_g-\#\mathrm{exc}(\pi)
$$
or equivalently,
\begin{equation} \label{pgmueq}
\frac{12(n-1)}{3n-5}p_g\le \mu+1-\chi_{\mathrm{top}}(\mathrm{exc}(\pi))
\end{equation}
with the equality holding if and only if $(X,o)$ is a rational double point,
where $\#\mathrm{exc}(\pi)$ is the number of irreducible components of the exceptional set $\mathrm{exc}(\pi)$ and $\chi_{\mathrm{top}}(\mathrm{exc}(\pi))$ its topological Euler number.
\end{thm}

\begin{proof}
Let $(X,o)$ be a germ of a $2$-dimensional isolated complete intersection singularity with embedding dimension $n+1\ge 3$.
Then $X$ can be regarded as the central fiber $h^{-1}(0)$ of a surjective holomorphic map $h=(h_1,\ldots,h_{n-1})\colon \mathbb{C}^{n+1}\to \mathbb{C}^{n-1}$ with $h(o)=0$, 
where by algebraization, we may assume that all $h_{i}$'s are polynomials in $\mathbb{C}[z_1,\ldots,z_{n+1}]$.
We can naturally regard $\mathbb{C}^{n+1}$ as a Zariski open subset of $\mathbb{P}^{n}\times \mathbb{P}^{1}$.
Let $\overline{X}$ be the closure of $X$ in $\mathbb{P}^n\times \mathbb{P}^1$.
Adding monomials of sufficiently higher order to the defining equations $h_i$, we may assume that $\overline{X}$ also has only one singularity $o$.
Let $(Z:Z_1:\cdots:Z_n)$ and $(Z':Z_{n+1})$ respectively denote the homogeneous coodinates of $\mathbb{P}^{n}$ and $\mathbb{P}^1$.
Then the defining equation $\overline{h}_i(Z,Z_1,\ldots,Z_n;Z',Z_{n+1})$ of $\overline{X}$ in $\mathbb{P}^n\times \mathbb{P}^1$ is the bi-homogenization of $h_i(z_1,\ldots,z_{n+1})$ for $i=1,\ldots,n-1$.
We may also assume that all $\overline{h}_i$'s have the same degree $d\gg 0$ with respect to $(Z:Z_1:\cdots:Z_n)$.
Let $\pi\colon S\to \overline{X}$ be the minimal resolution of $o\in \overline{X}$.
Let $\overline{f}=pr_2|_{\overline{X}}\colon \overline{X}\to \mathbb{P}^1$ be the second projection and $f:=\overline{f}\circ \pi\colon S\to \mathbb{P}^1$.
We may assume that $f$ is relatively minimal, since by adding sufficiently higher order terms to the defining equations, it follows that any component of the central fiber $f^{-1}(0)$ other than $\mathrm{exc}(\pi)$ is not a $(-1)$-curve.
We can easily check that the fibration $f\colon S\to B=\mathbb{P}^1$ satisfies the assumption of Theorem~\ref{Mainthm}.
Thus we get 
\begin{equation} \label{slopeeq}
K_f^2\ge \lambda_{n,d}\chi_f
\end{equation}
from Theorem~\ref{Mainthm}.
On the other hand, we have 
\begin{equation} \label{cycleeq}
K_{\overline{f}}^2=\lambda_{n,d}\chi_{\overline{f}}
\end{equation}
from Lemma~\ref{exalem}.
Let $p_g=\mathrm{dim}_{\mathbb{C}}R^{1}\pi_{*}\mathcal{O}_{S}$ be the geometric genus of $(X,o)$ and $K$ the canonical cycle on $\mathrm{exc}(\pi)$ (that is, $K=\sum_{i}b_iE_i$ is a cycle on the intersection form $\mathrm{exc}(\pi)=\cup_{i}E_i$ uniquely determined by the genus formula $2p_a(E_i)-2=(K+E_i)E_i$ for any $E_i$).
Thus we have 
\begin{equation} \label{pgeq}
-p_g=\chi_f-\chi_{\overline{f}}
\end{equation}
and 
\begin{equation} \label{K2eq}
K^2=K_f^2-K_{\overline{f}}^2.
\end{equation}
Subtracting \eqref{cycleeq} from \eqref{slopeeq} and using \eqref{pgeq} and \eqref{K2eq}, we have
$$
K^2+\lambda_{n,d}p_g\ge 0.
$$
Since 
$$
\lambda_{n,d}=\frac{24((n-1)d-(n+1))}{(3n-2)d-(3n+2)}\le \frac{24(n-1)}{3n-2},
$$
we have
\begin{equation} \label{K2pgineq}
K^2+\frac{24(n-1)}{3n-2}p_g\ge 0
\end{equation}
with the equality holding if and only if $(X,o)$ is a rational double point.
From Laufer's formula \cite{La}
$$
\mu=12p_g+K^2+\chi_{\mathrm{top}}(\mathrm{exp}(\pi))-1,
$$
Durfee's equality \cite{Dur}
$$
2p_g=\mu_{+}+\mu_0
$$
and
$$
\chi_{\mathrm{top}}(\mathrm{exc}(\pi))=\#\mathrm{exc}(\pi)+1-\mu_0,
$$
\eqref{K2pgineq} is equivalent to
$$
\sigma\le -\frac{8}{3n-2}p_g-\#\mathrm{exc}(\pi).
$$

\end{proof}

\end{document}